\documentclass[bibalpha]{amsart}

\title{Uniformly defining 
$p$-henselian valuations}

\author{Franziska Jahnke} 
\address{Institut f\"ur Mathematische Logik\\Einsteinstr. 62\\48149 M\"unster, 
Germany}
\email{franziska.jahnke@uni-muenster.de}

\author{Jochen Koenigsmann}
\address{Mathematical Institute\\Radcliffe Observatory Quarter\\
Woodstock Road\\Oxford OX2 6GG, UK}
\email{koenigsmann@maths.ox.ac.uk}

\thanks{
Some of the research leading to these results has received funding from the 
[European Community's] Seventh Framework Programme 
[FP7/2007-2013] under grant agreement 
number 238381.}


\usepackage{verbatim}
\usepackage{enumerate}
\usepackage{braket}
\usepackage{txfonts}
\newtheorem{Th}{Theorem}[section]
\newtheorem{Thm}[Th]{Theorem}
\newtheorem*{MainTheorem}{Main Theorem}
\newtheorem{MaiTheorem}[Th]{Main Theorem}
\newtheorem*{Def}{Definition}
\newtheorem{Cor}[Th]{Corollary}
\newtheorem{Prop}[Th]{Proposition}
\newtheorem*{Ex}{Example}
\newtheorem*{Rem}{Remark}
\newtheorem{Obs}[Th]{Observation}
\newtheorem{Lem}[Th]{Lemma}

\begin{document}
\begin{abstract}
Admitting a non-trivial $p$-henselian
valuation is a weaker assumption on a field than admitting a non-trivial henselian
valuation. Unlike henselianity, $p$-henselianity is an elementary property
in the language of rings. We are interested in the question when a field
admits a non-trivial 0-definable $p$-henselian valuation (in the language
of rings). We give a classification of elementary
classes of fields in which the canonical $p$-henselian valuation is
uniformly 0-definable.  We then apply this to show that there is a definable
valuation inducing the ($t$-)henselian topology on any ($t$-)henselian 
field which is neither separably closed nor real closed.
\end{abstract}

\maketitle
\section{Introduction}
Where a valuation $v$ on a field $K$
contributes to the arithmetic of $K$,
e.g., in the sense that the existence of $K$-rational points
on certain algebraic varieties defined over $K$
is guaranteed or prohibited by `local' conditions `at $v$',
the valuation $v$ (or rather its valuation ring ${\mathcal O}_v$)
is often definable by a first-order formula $\phi (x)$
in the language of rings ${\mathcal L}_\textrm{ring} = \{+,\times; 0,1\}$: 
For each $a\in K$, one has $a\in {\mathcal O}_v$
if and only if $\phi (a)$ holds in $K$
-- we then write ${\mathcal O}_v = \phi (K)$.

This happens, for example, for all valuations in all global fields
(a fact implicit in the pioneering works \cite{Ro49} and \cite{Ro59} 
of Julia Robinson),
and later, Rumely even found a {\em uniform} 
first-order definition for all valuation rings in all global fields 
(\cite{Ru80}).
It also happens in the classical henselian fields
$\mathbb{Q}_p$ and $\mathbb{F}_p((t))$
or $k((t))$ for an arbitrary field of coefficients $k$
via the well known formulas for $\mathbb{Z}_p$ in $\mathbb{Q}_p$
and for $k[[t]]$ in $k((t))$ due to Ax and others.
It does not happen on $\mathbb{C}$ or on $\mathbb{R}$
or on any algebraically or real closed field,
where no valuation is of arithmetical interest,
and where no non-trivial valuation is first-order definable,
because, by quantifier elimination,
first-order definable subsets of algebraically closed fields are finite or cofinite
and those on real closed fields are finite unions of intervals and points.

In the 1970's the concept of a {\em 2-henselian} valuation
emerged from the algebraic theory of quadratic forms,
and later, by way of analogy,
the notion of a $p$-henselian valuation was coined
for an arbitrary prime number $p$:
A valuation $v$ on a field $K$ is called {\em $p$-henselian}
if $v$ has a unique prolongation to $K(p)$,
the maximal Galois-$p$ extension of $K$
(i.e., the compositum of all finite Galois extensions of $p$-power degree
over $K$ in some fixed algebraic closure of $K$).
Equivalently, $v$ is $p$-henselian on $K$
if it has a unique prolongation to each Galois extension of degree $p$
-- this fact that $p$-henselianity shows in Galois extensions of bounded degree
makes it easier to find definable $p$-henselian valuations
compared to finding definable henselian valuations.
Note that every henselian valuation is $p$-henselian
but, in general, not the other way round.

Like for henselian valuations
there may be several $p$-henselian valuations on a field $K$,
but there always is a canonical one:
the {\em canonical $p$-henselian valuation $v_K^p$} on a field $K$
is the coarsest $p$-henselian valuation $v$ on $K$
whose residue field $Kv$ is $p$-closed
(i.e., where $Kv = Kv(p)$)
if there is any such;
if not it is the finest $p$-henselian valuation on $K$
(cf.\,section 3 of \cite{Koe95} 
where existence and uniqueness of $v_K^p$ is proven).
Recall that, for two valuations $v,w$ on $K$,
$v$ is finer than $w$ 
just in case ${\mathcal O}_v\subseteq {\mathcal O}_w$.
Recall further that if $v$ is finer than $w$, then,
equivalently, $w$ is coarser than $v$.
The valuation $v_K^p$ is non-trivial
if and only if $K$ admits a non-trivial $p$-henselian valuation.

This paper is intended to 
both close a gap in the proof of Theorem 3.2 of \cite{Koe95}
and to present a more uniform version of the Theorem.
This Theorem asserts that $v_K^p$ is first-order definable
if $K$ is of characteristic $p$
or if $K$ contains a primitive $p$-th root $\zeta_p$ of unity
and, if $p=2$, the residue field $Kv_K^p$ is not Euclidean.
The gap occurred in the case where
$(K,v_K^p)$ is of mixed characteristic $(0,p)$
(i.e., $\mathrm{char}(K)= 0$ and $\mathrm{char}Kv_K^p = p$). However, we
also present a slightly different proof to the (incomplete) one in
\cite{Koe95}.

To phrase the true definability result for $v_K^p$
we should also take care of cases
where $v_K^p$ is, as it were,
only definable `by accident',
that is, definable for the wrong reason.
For example, there might be another prime $q\neq p$
with $v_K^q = v_K^p$,
where $v_K^q$ is `truly' definable, but $v_K^p$ is not.
To pin this down
we say that $v_K^p$ is {\em $\emptyset$-definable as such}
if there is a parameter-free ${\mathcal L}_\textrm{ring}$-formula $\phi(x)$
such that, for all fields $L$ elementarily equivalent to $K$
in $\mathcal{L}_\textrm{ring}$
(which we denote by $L\equiv K$),
${\mathcal O}_{v_L^p} = \phi (L)$.
With this terminology we not only get
a precise criterion for true (= `as such') definability of $v_K^p$,
but also the most uniform definition of $v_K^p$ that one could wish for:
a single ${\mathcal L}_\textrm{ring}$-formula $\phi_p(x)$
does it for all of them:

\begin{MainTheorem}
For each prime $p$
there is a parameter-free ${\mathcal L}_\textrm{ring}$-formula $\phi_p(x)$
such that for any field $K$ with 
either $\mathrm{char}(K) = p$ or $\zeta_p\in K$
the following are equivalent:
\begin{enumerate}
\item 
$\phi_p$ defines $v_K^p$ as such.
\item $v_K^p$ is $\emptyset$-definable as such.
\item $p\neq2$ or $Kv_K^p$ is not Euclidean.
\end{enumerate}
\end{MainTheorem}

The paper is organized as follows. We recall well-known 
definitions and facts about $p$-henselian valuations in the second section. 
In the third section, we give our Main Theorem and draw some conclusions 
from
it. The Main Theorem is then proven in section 4. Finally, we apply the Main 
Theorem to $t$-henselian fields in the last section. Improving a result of 
Koenigsmann (Theorem 4.1 in \cite{Koe94}), we show that any
$t$-henselian field which is neither separably closed nor real closed admits a
definable valuation inducing the (unique) $t$-henselian topology.

\section{p-henselian valuations and their canonix}
Throughout this section, let $K$ be a field and $p$ a prime. We use
the following notation: If $v$ is a valuation on $K$, we write
$\mathcal{O}_v$ for the valuation ring, $\mathfrak{m}_v$ for the maximal
ideal, $Kv$ for the residue field and $vK$ for the value group of $(K,v)$.
For $a \in \mathcal{O}_v$, we use $\bar{a}$ to denote its image in $Kv$.
\begin{Def}
We define $K(p)$ to be the compositum of all Galois extensions of $K$ of 
$p$-power degree. 
A valuation $v$ on $K$ is called
\emph{$p$-henselian} if $v$ extends uniquely to $K(p)$.
We call $K$ \emph{$p$-henselian} if $K$ 
admits a non-trivial
$p$-henselian valuation. 
\end{Def}
Clearly, this definition only imposes a condition on $v$ if $K$ admits
Galois extensions
of $p$-power degree. 

\begin{Prop}[\cite{Koe95}, Propositions 1.2 and 1.3] \label{phenseq}
For a valued field $(K,v)$, the following are equivalent:
\begin{enumerate}
\item $v$ is $p$-henselian,
\item $v$ extends uniquely to every Galois extension of $K$ of $p$-power
degree,
\item $v$ extends uniquely to every Galois extension of $K$ of degree
$p$,
\item for every polynomial $f \in {\mathcal O}_v$ which splits in $K(p)$ and
every $a \in {\mathcal O}_v$ with $\bar{f}(\overline{a}) = 0$ and
$\bar{f'}(\overline{a}) \neq 0$, there exists $\alpha
\in {\mathcal O}_v$ with $f(\alpha)=0$ and $\overline{\alpha}=\overline{a}$.
\end{enumerate}
\end{Prop}

As for fields carrying a henselian valuation, there is a canonical
$p$-henselian valuation:

\begin{Thm}[\cite{Br76}, Corollary 1.5]
If $K$ carries two independent non-trivial 
$p$-hen\-se\-lian valuations, then $K = K(p)$.
\end{Thm}

Assume that $K \neq K(p)$. We divide the class of 
$p$-henselian valuations on $K$
 into
two subclasses,
$$H^p_1(K) = \Set{v\; p\textrm{-henselian on } K | Kv \neq Kv(p)}$$
and
$$H^p_2(K) = \Set{ v\; p\textrm{-henselian on } K | Kv = Kv(p) }.$$ 
One can deduce that any valuation $v_2 \in H^p_2(K)$ 
is \emph{finer} than any $v_1 \in H^p_1(K)$, i.e. 
${\mathcal O}_{v_2} \subsetneq {\mathcal O}_{v_1}$,
and that any two valuations in $H^p_1(K)$ are comparable.
Furthermore, if $H^p_2(K)$ is non-empty, then there exists a unique coarsest
valuation
$v_K^p$ in $H^p_2(K)$; otherwise there exists a unique finest 
valuation $v_K^p \in H^p_1(K)$.
In either case, $v_K^p$ is called the \emph{canonical $p$-henselian valuation}.
If $K$ is $p$-henselian then $v_K^p$ is non-trivial.

\section{The main theorem and some consequences}
We want to find a uniform definition of
the canonical $p$-henselian valuation. As $p$-henselianity is an elementary
property, any sufficiently 
uniform definition of $v_K^p$ on some field $K$ will also define the
canonical $p$-henselian valuation in any field elementarily equivalent to $K$.
This motivates the following
\begin{Def}
Let $K$ be a field, assume that
$K \neq K(p)$ and that $\zeta_p \in K$ in case $\mathrm{char}(K)\neq p$.
We say that $v_K^p$ is \emph{$\emptyset$-definable as such} 
if there is a parameter-free
$\mathcal{L}_\textrm{ring}$-formula $\phi_p(x)$ such that
$$\phi_p(L) = \mathcal{O}_{v_L^p}$$
holds in any $L \equiv K$.
\end{Def}

Recall that a field $F$ is called Euclidean if $[F(2):F]=2$. This is an 
elementary property in $\mathcal{L}_\textrm{ring}$: Every Euclidean field is uniquely
ordered, the positive elements being exactly the squares. By the well-known
results of Artin-Schreier/Becker (cf.\,\cite[Theorem 4.3.5]{EP05}), 
Euclidean
fields are the only fields for which $F(p)$ can be a proper finite extension of
$F$.

We are now in a position to state our main theorem:
\begin{MaiTheorem} \label{main}
Fix a prime $p$. There exists a parameter-free 
$\mathcal{L}_\textrm{ring}$-formula $\phi_p(x)$
such that for any field $K$ with
either $\mathrm{char}(K)=p$ or $\zeta_p \in K$
the following are equivalent:
\begin{enumerate}
\item $\phi_p$ defines ${v_K^p}$ as such. 
\item $v_K^p$ is $\emptyset$-definable as such.
\item $p\neq 2$ or $Kv_K^p$ is not Euclidean.
\end{enumerate}
\end{MaiTheorem}
Note that it may well happen that $v_K^p$ is definable, but not 
definable as such: 

\begin{Ex}
Consider the field $K = \mathbb{R}((t))$. Then the canonical 
$2$-henselian valuation coincides with the power series valuation as 
$\mathbb{R}$ is not ($2$-)henselian. Furthermore, we have $v_K^p=v_K^2$
for all primes $p$. In particular, $v_K^2$ is 
$\emptyset$-definable, 
say via the $\mathcal{L}_\textrm{ring}$-formula $\phi(x)$. Now, we have
$$K =\mathbb{R}((t)) \equiv \mathbb{R}((s^{\mathbb{Q}}))((t)) 
=: L$$
by the Ax-Kochen/Ersov Theorem (\cite[Theorem 4.6.4]{PD}) 
since $\mathbb{R}((s^{\mathbb{Q}}))$
is real closed. However, $\phi(L)$ defines a henselian valuation on $L$
with value group elementarily equivalent to $\mathbb{Z}$. Thus, we get
$$\mathcal{O}_{v_L^2}  \subsetneq 
\phi(L) = \mathbb{R}((s^{\mathbb{Q}}))[[t]],$$ 
as the canonical $2$-henselian valuation on $L$ has residue field $\mathbb{R}$.
Hence, $v_K^2$ is $\emptyset$-definable but not $\emptyset$-definable as such.
\end{Ex}

Before we prove the theorem, we draw some conclusions from it.
\begin{Obs} \label{2def}
Let $K \neq K(2)$, and assume that $Kv_K^2$ is Euclidean.
Then 
\begin{enumerate}[(a)]
\item the coarsest $2$-henselian valuation $v_K^{2*}$ on $K$ 
which has Euclidean
residue field is $\emptyset$-definable and 
\item there is an
$\mathcal{L}_\textrm{ring}$-sentence $\epsilon$ such that for any field $K$
$$K \models \epsilon \Longleftrightarrow K \textrm{ is not Euclidean and } 
Kv_K^2 \textrm{ is Euclidean.}$$
\end{enumerate}
\end{Obs}
\begin{proof} 
\begin{enumerate}[(a)] 
\item We amend the proof of 3.2 in \cite{Koe95} to our needs. 
As $Kv_K^2$ is Euclidean,
all $2$-henselian valuations are comparable and 
coarsenings of $v_K^2$. If $K$ is Euclidean, then the coarsest $2$-henselian
valuation with Euclidean residue field is the trivial one and thus 
$\emptyset$-definable. Else, $K$ is not 
Euclidean and there is a coarsest (non-trivial) 
$2$-henselian valuation $v_K^{2*}$ such that
$Kv_K^{2*}$ is Euclidean. 

We use Beth's Definability Theorem to show that $v_K^{2*}$ is definable. 
If we add a symbol for $\mathcal{O}_v$ to the ring language, then 
we claim that $\mathcal{O}_v={\mathcal O}_{v_K^{2*}}$ 
is axiomatized by the properties
\begin{enumerate}[(i)]
\item $v$ is $2$-henselian,
\item $Kv$ is Euclidean,
\item no non-trivial convex subgroup of $vK$ is 
$2$-divisible. This is an elementary property of the ordered abelian group 
$vK$ (and thus of the valued field $(K,v)$ in 
$\mathcal{L}_{val}=\mathcal{L}_{ring} \cup\{ \mathcal{O}_v\}$) as it is axiomatized by the sentence
$$\forall \alpha \,( \alpha > 0 \rightarrow \exists \gamma\forall \delta \, 
(0 < \gamma \leq \alpha\, \wedge \, \gamma \neq \delta+\delta)).$$   
\end{enumerate}
Clearly, $v_K^{2*}$ satisfies the first two of these axioms. Furthermore, since
$Kv_K^{2*}$ is Euclidean but $K$ is not, $v_K^{2*}K$ is not $2$-divisible.
Note that for any $2$-henselian valuation with Euclidean residue field,
$2$-divisible convex subgroups correspond exactly to coarsenings
with Euclidean residue field.
Thus, as every proper coarsening of $v_K^{2*}$ has non-Euclidean residue field, 
$v_K^{2*}K$ has no non-trivial convex $2$-divisible subgroups.
Since all $2$-henselian valuations are comparable, 
$v_K^{2*}$ is the only $2$-henselian valuation with Euclidean residue 
field and value group having no non-trivial $2$-divisible convex
subgroup, hence 
it is indeed characterized by these properties.

As the same characterization gives $v_L^{2*}$ in any $L \equiv K$ and
no parameters were needed, $v_K^{2*}$ is 
$\emptyset$-definable by Beth's Definability Theorem 
(see \cite{Hod97}, Theorem 5.5.4), say via the 
$\mathcal{L}_\textrm{ring}$-formula $\eta(x)$.

\item The sentence $\epsilon$ is now given by
\begin{align*}
\epsilon\, \equiv\; & K \textrm{ not Euclidean and }\\
&\eta(x) \textrm{ defines a }2\textrm{-henselian 
valuation ring }\mathcal{O}_v \subseteq K
\textrm{ with } Kv \textrm{ Euclidean.}
\end{align*}
\end{enumerate}
\end{proof}

We can now also give a version of Theorem \ref{main} which includes the 
special case of Euclidean residue fields:
\begin{Cor} \label{Cmain}
Let $p$ be a prime and 
consider the (elementary) class of fields
$$\mathcal{K} = \Set{ K | K \;p\textrm{-henselian, with }\zeta_p \in K
\textrm{ in case } \mathrm{char}(K)\neq p}$$
There is a parameter-free $\mathcal{L}_\textrm{ring}$-formula $\psi_p(x)$
such that 
\begin{enumerate}
\item if $p \neq 2$ or $Kv_2$ is not Euclidean, then $\psi_p(x)$ defines 
the valuation ring of the canonical
$p$-henselian valuation $v_K^p$, and
\item if $p=2$ and $Kv_2$ is Euclidean, then $\psi_p(x)$ defines the valuation
ring of the
coarsest $2$-henselian valuation $v_K^{2*}$ such that
$Kv_K^{2*}$ is Euclidean.
\end{enumerate}
\end{Cor}
\begin{proof}
Fix a prime $p$. If $p \neq 2$, then the statement follows immediately from
Theorem \ref{main}. Now assume $p=2$.
By Corollary 2.2 in \cite{Koe95} and Observation \ref{2def}, the classes
\begin{align*}
\mathcal{K}_1 =& \Set{ K \textrm{ is }2\textrm{-henselian and }Kv_K^2 
\textrm{ is not Euclidean} } \textrm{ and }\\
\mathcal{K}_2 =& \Set{ K \textrm{ is }2\textrm{-henselian and }Kv_K^2 
\textrm{ is Euclidean} }
\end{align*}
are both elementary in $\mathcal{L}_\textrm{ring}$. 
Thus, using Theorem \ref{main} and Observation \ref{2def}, 
there is a parameter-free
formula $\psi_2$
defining $v_K^2$ in any $K \in \mathcal{K}_1$ and $v_K^{2*}$ in any 
$K \in \mathcal{K}_2$ by Beth's Definability Theorem
(Theorem 5.5.4 in \cite{Hod97}).
\end{proof}

\section{Proof of the main theorem}
\begin{Lem} \label{1}
Let $(K,v)$ be a non-trivially valued field with $\zeta_p \in K$ in case
$\mathrm{char}(K)\neq p$ such that one of
\begin{enumerate}
\item $vK$ has no non-trivial convex $p$-divisible subgroup or
\item $\mathrm{char}(Kv)=p$ and $Kv$ is not perfect
\end{enumerate}
holds. Then we have $Kw \neq Kw(p)$ for any proper coarsening $w$ of $v$.
\end{Lem}
\begin{proof}
Let $w$ be a proper coarsening of $v$ and let $\Delta \leq vK$ be the
non-trivial convex subgroup of $vK$ with $wK=vK/\Delta$. Then
$v$ induces a valuation $\bar{v}$ on $Kw$ with residue field 
$(Kw)\bar{v}=Kv$ and
value group $\bar{v}(Kw)= \Delta$.

We first assume condition $(1)$. Then, we have $\Delta \neq p\Delta$, 
thus there
is some $x \in Kw$ such that $\bar{v}(x) \notin p \Delta$.
If $\mathrm{char}(Kw) \neq p$, then $\zeta_p \in Kw$ and so $Kw(\sqrt[p]{x})$
is a Galois extension of $Kw$ of degree $p$. In particular, we get
$Kw \neq Kw(p)$.

If $\mathrm{char}(Kw)=p$, we may (after possibly replacing $x$ by $x^{-1}$)
assume that $\bar{v}(x) <0$. In this case, the polynomial $X^p-X-x$
has no zero in $Kw$: Any zero $\alpha$ would satisfiy 
$\bar{v}(\alpha)=\frac{1}{p}\bar{v}(x)$. Hence, $Kw(\alpha)$ is again a 
Galois extension of $Kw$ of degree $p$ and so $Kw \neq Kw(p)$ as
required.

Now, we assume condition $(2)$. As $(Kw)\bar{v}=Kv$ is not perfect, 
we can choose some
$a \in \mathcal{O}_{\bar{v}}^\times \subseteq Kw$ with 
$\bar{a}=a + \mathfrak{m}_{\bar{v}} \notin (Kv)^p$. If $\mathrm{char}(Kw) 
\neq p$, then
$a \notin (Kw)^p$ and $Kw(\sqrt[p]{a})$ is a Galois extension of $Kw$
of degree $p$ as before.
If $\mathrm{char}(Kw)=p$, we pick some $x \in \mathfrak{m}_{\bar{v}}$
and observe that any zero $\alpha$ of $X^p-X-ax^{-p}$ satisfies
$\bar{v}(\alpha)=\bar{v}(x^{-1})$. This implies $(\alpha x)^p-a=
\alpha x^p \in \mathfrak{m}_{\bar{v}}$, hence we get
$(\overline{\alpha x})^p = \bar{a}$ (modulo $\mathfrak{m}_{\bar{v}}$). 
Thus, $Kw(\alpha)/Kw$ is once more 
a Galois extension
of degree $p$.
\end{proof}

\begin{Cor} \label{1.1}
Let $(K,v)$ be $p$-henselian such that
$\zeta_p \in K \neq K(p)$, $\mathrm{char}(Kv) \neq p$ 
and $Kv = Kv(p)$ hold.
Then
$$v = v_K^p \Longleftrightarrow vK \textrm{ has no non-trivial $p$-divisible
convex subgroup}.$$
\end{Cor}
\begin{proof}
If $vK$ has no non-trivial $p$-divisible convex subgroup, then $v = v_K^p$
follows from the definition of $v_K^p$ and \ref{1}.

Conversely, assume that 
$vK$ has a non-trivial $p$-divisible convex subgroup $\Delta$,
and let $w$ be the corresponding 
coarsening of $v$ with value group $wK = vK/\Delta$.
We show that $Kw = Kw(p)$ holds. Note that the valuation 
$\bar{v}$ induced by $v$ on $Kw$ is $p$-henselian
(\cite{EP05}, Corollary 4.2.7), 
has residue field $Kv= Kv(p)$
and $p$-divisible value group.
As $\mathrm{char}(K) \neq p$, this implies $Kw= (Kw)^p$ and thus $Kw=Kw(p)$.
Hence, $v$ is a proper refinement of $v_K^p$.
\end{proof}

We next treat the case $\mathrm{char}(K)=p$.
\begin{Lem} \label{2}
Let $(K,v)$ be a $p$-henselian valued field such that $\mathrm{char}(K)=p$ 
and $Kv= Kv(p)$ holds. Then
\begin{align*} v = 
v_K^p &\Longleftrightarrow \forall x \in \mathfrak{m}_v \setminus \{0\}: 
\;
x^{-1}\mathcal{O}_v \not\subseteq K^{(p)} = \Set{ x^p-x | x \in K}.
\end{align*}
\end{Lem}
\begin{proof}
We may assume $K\neq K(p)$ since otherwise $v_K^p$ is trivial and 
there is nothing to prove.

The implication from left to right is shown in the first part of the
 proof of Theorem 3.1 in \cite{Koe94}. As it is unpublished,
we repeat the proof for the convenience of the reader.
The idea of the proof is as follows: If $v$ is a
$p$-henselian valuation on a field $K$ with $\mathrm{char}(K)=p$
satisfying $Kv=Kv(p)$, then $p$-Hensel's lemma (Proposition
\ref{phenseq}) implies 
$\mathcal{O}_v \subseteq K^{(p)}$. We show that if
$Kv_K^p=Kv_K^p(p)$, then no
proper coarsening
of $v_K^p$ can have this property. This is exactly what the
statement
$$v = v_K^p \implies \forall x \in \mathfrak{m}_v \setminus \{0\}: 
\;
x^{-1}\mathcal{O}_v \not\subseteq K^{(p)} = \Set{ x^p-x | x \in K}$$
expresses.

Take $x \in \mathfrak{m}_v \setminus \{0\}$
satisfying $x^{-1} \mathcal{O}_v \subseteq K^{(p)}$. Then, there is a 
maximal fractional ideal $\mathfrak{N}$ of $\mathcal{O}_v$
such that $\mathcal{O}_v \subsetneq  \mathfrak{N} \subseteq K^{(p)}$
holds.
Consider the convex hull $\Gamma$ of the group generated by 
$$\Set{v(a) | a \in \mathfrak{N}\setminus \mathcal{O}_v}$$ in $vK$.
Then, $\Gamma$ is a non-trivial $p$-divisible subgroup of $vK$. Assume now
that also $v=v_K^p$ holds. We now prove that this contradicts the maximality
of $\mathfrak{N}$.

As all coarsenings $v'$ of $v$ satisfy $Kv' \neq Kv'(p)$,
$\mathfrak{N}$ contains no proper coarsenings of $\mathcal{O}_v$.
Hence $\Gamma$ is archimedean. Note that for any $\alpha \in \mathbb{R}$ with
$\alpha > 1$, the fractional $\mathcal{O}_v$-ideal
$$ \mathfrak{N}_\alpha \coloneqq \Set{ z \in K \mid 
v(z) \geq \alpha \cdot v(y) \textrm{ for some }y \in \mathfrak{N}}$$
strictly contains $\mathfrak{N}$. To get a contradiction, we now choose any
$\alpha \in \; ]1,2-\frac{1}{p}[$ and show that $\mathfrak{N}_\alpha$
is contained in $K^{(p)}$. 
Choose any $z \in \mathfrak{N}_\alpha \setminus
\mathfrak{N}$, say $v(z) \geq \alpha \cdot v(y)$ for some 
$y \in  \mathfrak{N}$. In particular, $v(z)<0$.
Then, as $1 < \alpha < 2$, we have
$0>v(zy^{-1})>v(y)$ and thus $zy^{-1} \in  \mathfrak{N} \setminus \mathcal{O}_v$.
Hence, there is some $a \in K \setminus \mathcal{O}_v$ with
$zy^{-1} \in a^p\mathcal{O}_v^\times$ and thus $za^{-p} \in \mathfrak{N}$.
Moreover, there is some $b \in K \setminus \mathcal{O}_v$ with 
$za^{-p} = b^p-b$. Now, we have
$z=a^pb^p-a^pb$, but 
\begin{align*}
v(a^pb) &= v(zy^{-1})+v(b)\\
&=v(z)-v(y)+\dfrac{1}{p}v(y)\\
&\geq (\alpha-1+\dfrac{1}{p})\cdot v(y)\\
&\geq v(y)
\end{align*}
so $v(ab) > v(a^pb) \geq v(y)$ holds. 
This implies $ab,a^pb\in \mathfrak{N} \subseteq K^{(p)}$.
As $K^{(p)}$ is closed under addition, we conclude 
$z=(ab)^p-(ab)+ab-a^pb \in K^{(p)}$, contradicting
the maximality of $\mathfrak{N}$.

Suppose now that we have $v \neq v_K^p$. Then, by the definition of $v_K^p$,
we get
$\mathcal{O}_v \subsetneq\mathcal{O}_{v_K^p}$ and $Kv_K^p = Kv_K^p(p)$.
By Proposition \ref{phenseq}, the polynomial $X^p-X-a$ has a zero for any 
choice of
$a \in \mathcal{O}_{v_K^p}$, thus $\mathcal{O}_{v_K^p} \subseteq K^{(p)}$
holds. 
Now for any $x \in \mathcal{O}_{v_K^p}^\times \cap \mathfrak{m}_v$, 
we have
$$x^{-1}  \mathcal{O}_{v} \subseteq  x^{-1} \mathcal{O}_{v_K^p}
=  \mathcal{O}_{v_K^p} \subseteq K^{(p)},$$
so the condition on the right does not hold.
\end{proof}

\begin{Lem}[Lemma 3.2 in \cite{Koe03}] \label{Fact}
Let $(K,v)$ be $p$-henselian of characteristic $(0,p)$ with $\zeta_p \in K$.
Then for any $a \in \mathcal{O}_v$ we have
$$1+(1-\zeta_p)^pa \in (K^\times)^p \Longleftrightarrow \exists x \in Kv: \;
x^p-x-\bar{a}=0.$$
\end{Lem}

We now consider the case $(\mathrm{char}(K),\mathrm{char}(Kv)=(0,p)$.
The proof is similar to the one of Lemma \ref{2}, it is
however even more technical. 
\begin{Lem} \label{3}
Let $(K,v)$ be a $p$-henselian valued field with $\mathrm{char}(K)=0$ and 
$\mathrm{char}(Kv)=p$.
Assume further that $Kv$ is perfect, $vK$ has a non-trivial $p$-divisible
convex subgroup and that $Kv= Kv(p)$ holds. Then
\begin{align*} v = 
v_K^p 
&\Longleftrightarrow \forall x \in \mathfrak{m}_v \setminus \{0\}: 
\; 1+x^{-1}(\zeta_p-1)^p\mathcal{O}_v \not\subseteq (K^\times)^p. 
\end{align*}
\end{Lem}
\begin{proof}
We may assume $K\neq K(p)$ since otherwise $v_K^p$ is trivial and 
there is nothing to prove.

The idea of the proof is very similar to the one of
Lemma \ref{2}: We want to show that 
$v_K^p$ is characterized by being the coarsest $p$-henselian valuation
$v$ which satisfies 
$ 1+(\zeta_p-1)^p\mathcal{O}_v \subseteq (K^\times)^p.$ The fact that
this condition holds for $v=v_K^p$ under the assumptions of the lemma
follows from Lemma \ref{Fact}. To see that no proper coarsening of $v_K^p$
satisfies this, we again construct a maximal fractional ideal 
$\mathfrak{N}$ of $\mathcal{O}_{v_K^p}$ with 
$1+(\zeta_p-1)^p  \mathfrak{N} \subseteq (K^\times)^{p}$. We then
show that $\mathfrak{N}$ contains no 
proper coarsenings of $\mathcal{O}_{v_K^p}$
and use this to contradict the maximality of $\mathfrak{N}$.

We show the direction from left to right first. 
Let $v$ be a $p$-henselian valuation on $K$ as in the assumption of
the Lemma.
Assume
that some proper coarsening $w$ of $v$ has residue characteristic $p$.
We now get
\begin{align*} v = v_K^p &\Longrightarrow 
\bar{v}=v_{Kw}^p \\
& \overset{\ref{2}}{\Longrightarrow} 
\forall y \in \mathfrak{m}_{\bar{v}} \setminus\{0\}: \; 
y^{-1}\mathcal{O}_{\bar{v}}\not\subseteq Kw^{(p)}\\
&  \overset{\ref{Fact}}{\Longrightarrow} \forall y \in \mathfrak{m}_{v}\setminus\{0\}: \; 
1 + y^{-1}(1-\zeta_p)^p\mathcal{O}_{v} = 
1 + y^{-1}(\zeta_p-1)^p\mathcal{O}_{v}\not\subseteq (K^\times)^{p},
\end{align*}
hence the implication from left to right 
as stated in the Lemma holds in this case.

Thus, for the remainder of the proof, we
may assume that $v_K^p$
has only coarsenings of residue characteristic $0$. We claim 
that if $v$ satisfies the assumption of the Lemma and we have
\begin{align}
1+ \mathfrak{m}_v \subseteq (K^{\times})^p \label{try1}
\end{align}
then we get $v \neq v_K^p$.
Consider the $p$-henselian coarsening $w$ of $v$
which is obtained by dividing out the maximal
convex $p$-divisible subgroup of $vK$ which is non-trivial by assumption. 
If $\mathrm{char}(Kw)\neq 0$, then we have $v\neq v_K^p$
by our assumption that all proper coarsenings
of $v_K^p$ have residue characteristic $0$. 
Hence, we may assume 
$\mathrm{char}(Kw)=0$.
Then, the valuation $\bar{v}$ induced
by $v$ on $Kw$ is $p$-henselian, has $p$-divisible value group and 
perfect residue field. Note that we also have 
$1+\mathfrak{m}_{\bar{v}} \subseteq (Kw^\times)^p$.
Take any $a \in Kw$. Then, there is some $b \in Kw$ with 
$\bar{v}(a)=\bar{v}(b^p)$, so we have 
$ab^{-p} \in \mathcal{O}_{\bar{v}}^\times$. As $Kv = (Kw)\bar{v}$ 
is perfect, there is some
$c \in Kw$ with $ab^{-p} \in c^p (1+\mathfrak{m}_{\bar{v}})$, hence we get
$a \in Kw^p$. Thus, $Kw$ is $p$-closed, implying $v\neq v_K^p$.  

Now, assume for a contradiction that for $v=v_K^p$, 
there is some
 $x \in \mathfrak{m}_v \setminus \{0\}$
such that 
\begin{align}
1+x^{-1}(\zeta_p-1)^p \mathcal{O}_v \subseteq (K^{\times})^p \label{basic} 
\end{align}
holds. 

Note that we have $1+y^{-1}(\zeta_p-1)^p \mathcal{O}_v \subseteq (K^{\times})^p$
for any $y \in \mathfrak{m}_v$ with $v(y) \leq v(x)$. Thus,
there is a 
maximal fractional ideal $\mathfrak{N}$ of $\mathcal{O}_v$
such that $$1+(\zeta_p-1)^p  \mathfrak{N} \subseteq (K^\times)^{p}$$ holds:
As $0 \notin \left(K^\times\right)^p$, we have 
$v(y) > v(\zeta_p-1)^p$ for all $y \in \mathfrak{N}$ with 
$y^{-1} \in \mathfrak{m}_v$. 
Hence, we have $$\mathfrak{N} \subseteq (\zeta_p-1)^{-p} 
\mathfrak{m}_v \subseteq  (\zeta_p-1)^{-p} 
\mathcal{O}_v. $$
By (\ref{try1}), we get $1+\mathfrak{m}_v \not\subseteq (K^\times)^p$
and hence
$$\mathfrak{N} \subsetneq (\zeta_p-1)^{-p} 
\mathfrak{m}_v.$$
Let $\Gamma$ be the convex hull of the group generated by 
$$\Set{v(a) | a \in \mathfrak{N}\setminus \mathcal{O}_v}$$ in $vK$.
Then $\Gamma$ is a non-trivial $p$-divisible subgroup of $vK$: 
For any $a \in \mathfrak{N} \setminus \mathfrak{m}_v$, we have
\begin{align*} 1+a (\zeta_p-1)^p &= (1+b)^p \textrm{ for some }b \in 
\mathfrak{m}_v\\
&=1+pb + \dotsc + pb^{p-1}+b^p.
\end{align*}
Thus, we get $v(b)<v(\zeta_p-1)$ and so $v(a(\zeta_p-1)^p)=v(b^p)$ holds.
Hence, we conclude that $\Gamma$ is $p$-divisible. 

Note that $\mathfrak{N}$ contains no proper coarsening $\mathcal{O}_w$ 
of $\mathcal{O}_v$:
Otherwise, we have $\mathfrak{m}_v \subseteq \mathcal{O}_w$ and thus
$$1 + \mathfrak{m}_v \subseteq 1+\mathcal{O}_w \subseteq (K^\times)^p$$
holds. We have already shown that this contradicts $v=v_K^p$, cf.~(\ref{try1}).

Hence, $\Gamma$ is archimedean. For any $\alpha \in \mathbb{R}$ with
$\alpha > 1$, the fractional $\mathcal{O}_v$-ideal
$$ \mathfrak{N}_\alpha \coloneqq \Set{ z \in K \mid 
v(z) \geq \alpha \cdot v(y) \textrm{ for some }y \in \mathfrak{N}}$$
strictly contains $\mathfrak{N}$. 
To get a contradiction, we now choose any 
$\alpha \in \; ]1,2-\frac{1}{p}[$ such that 
$\mathfrak{N}_\alpha \subseteq (\zeta_p-1)^{-p} \mathfrak{m}_v$ holds.
Pick $z \in \mathfrak{N}_\alpha \setminus
\mathfrak{N}$, say $v(z) \geq \alpha \cdot v(y)$ for some 
$y \in  \mathfrak{N}$. Then, as $1 < \alpha < 2$, we have
$0>v(zy^{-1})>v(y)$ and thus $zy^{-1} \in  \mathfrak{N} \setminus 
\mathcal{O}_v$.
Hence, there is some $a\in K \setminus \mathcal{O}_v$
 with
$$ zy^{-1} \in a^p\mathcal{O}_v^\times.$$
Moreover, as we have $za^{-p}\in \mathfrak{N}$, there
is some $b \in \mathfrak{m}_v$ with
\begin{align} 
1+za^{-p}(\zeta_p-1)^p = (1+b)^p. \label{star2}
\end{align}
Thus, we get
$$z(\zeta_p-1)^p = a^p(b^p +pb^{p-1} + \dotsc +pb)$$
and 
\begin{align}1+z(\zeta_p-1)^p = 1 + (ab)^p + p a^pb^{p-1} + \dots + pa^pb.
\label{Nummer}
\end{align}
\emph{Claim: } We have $pa^pb \in (\zeta_p-1)^p \mathfrak{N}$. \\
\emph{Proof of Claim:} In order to show the claim,
we may equivalently show $a^pb \in 
(\zeta_p-1) \mathfrak{N}$.
Note that we have $a^p \in (zy^{-1})\mathcal{O}_v^\times
\subseteq \mathfrak{N}$. Hence, in case $v(b) \geq v(\zeta_p-1)$, we get
immediately $a^pb\in (\zeta_p-1)\mathfrak{N}$.
If $v(b)<v(\zeta_p-1)$, then $v(b^p) < v(pb)$ and so we get by the equation
(\ref{star2}) that 
\begin{align} v(za^{-p}(\zeta_p-1)^p)=v(b^p) \label{noch}
\end{align}
holds. 
Note that 
\begin{align*}
v(a^pb(\zeta_p-1)^{-1})
 &= v(a^{p-1}a b(\zeta_p-1)^{-1})\\ 
&\overset{(\ref{noch})}{=} 
v(a^{p-1}a 
\sqrt[p]{z}(a^{-1})(\zeta_p-1)(\zeta_p-1)^{-1})\\
&= v(a^{p-1}\sqrt[p]{z})
\end{align*}
holds, so $a^{p-1}\sqrt[p]{z} \in \mathfrak{N}$ implies 
$a^pb \in (\zeta_p-1) \mathfrak{N}$.
As we have $y \in \mathfrak{N}$,
it now suffices to show that
$$v(a^{p-1}\sqrt[p]{z}) \geq v(y)$$
holds.
Since $\alpha \leq 2-\frac{1}{p}$, we calculate
\begin{align*} 
& v(z) \geq \alpha v(y) \geq (2-\frac{1}{p})v(y)\\
\Longrightarrow\; & v(z^p) \geq v(y^{2p-1})\\
\Longrightarrow\; & v((zy^{-1})^{p-1}z) \geq v(y^{p})\\
\Longrightarrow\; & v((a^p)^{p-1}z) \geq v(y^{p})\\
\Longrightarrow\; & v(a^{p-1}\sqrt[p]{z}) \geq v(y).
\end{align*}
This proves the claim.

Now, note that by the choice of $\alpha$, we have 
$z(\zeta_p-1)^p \in \mathfrak{m}_v$
and so, by equation (\ref{star2}), $(ab)^p \in \mathfrak{m}_v$ and thus
$ab \in \mathfrak{m}_v$. Hence, using the claim above, we get
\begin{align*}
1+z(\zeta_p-1)^p &= (1 + ab)^p -pab - \dots -p(ab)^{p-1}+pa^pb^{p-1}
+ \dots + pa^pb\\
&\in (1 +ab)^p +(\zeta_p-1)^p \mathfrak{N} \\
&\subseteq (1 +ab)^p (1+(\zeta_p-1)^p \mathfrak{N}) \subseteq (K^\times)^p
\end{align*}
This implies $$1+(\zeta_p-1)^p\mathfrak{N}_\alpha \subseteq (K^\times)^p,$$
contradicting the maximality of $\mathfrak{N}$.
Thus, we conclude
\begin{align*}  v=v_K^p
\Longrightarrow
\forall x \in \mathfrak{m}_v \setminus \{0\}: 
\; 1+x^{-1}(\zeta_p-1)^p\mathcal{O}_v \not\subseteq (K^\times)^p.
\end{align*}

Conversely, assume that
\begin{align} \forall x \in \mathfrak{m}_v \setminus \{0\}: 
\; 1+x^{-1}(\zeta_p-1)^p\mathcal{O}_v \not\subseteq (K^\times)^p \label{assu}
\end{align}
holds.

We first assume that all proper
coarsenings of $v$ have residue characteristic $0$. 
Let $w$ be a proper coarsening of $v$.
As $p \in \mathfrak{m}_v$, assumption (\ref{assu})
ensures that there is some $a \in \mathcal{O}_v$ such that 
we have
$$1 + \dfrac{1}{p}(\zeta_p-1)^pa \notin (K^\times)^p.$$
Since $\mathrm{char}(Kw)=0$,
we have $\mathcal{O}_v[\frac{1}{p}] \subseteq \mathcal{O}_w$
and thus $\frac{1}{p}(\zeta_p-1)^pa \in \mathcal{O}_w$.
However,
 $p$-henselianity (see Proposition \ref{phenseq}) implies
$$1+ \mathfrak{m}_w \subseteq (K^\times)^p.$$
In particular, we get $\frac{1}{p}(\zeta_p-1)^pa \in \mathcal{O}_w^\times$
and thus $Kw\neq Kw(p)$. Therefore, by the definition of the canonical
$p$-henselian valuation, we have $v = v_K^p$.

Otherwise, $v$ has a proper coarsening $w$ of residue characteristic $p$. 
Let $\bar{v}$ denote the valuation induced by $v$ on $Kw$. Then,
\begin{align*}
v = v_K^p &\Longleftrightarrow 
\bar{v}=v_{Kw}^p \\
& \overset{\ref{2}}{\Longleftrightarrow} 
\forall x \in \mathfrak{m}_{\bar{v}}\setminus\{0\}: \; 
x^{-1}\mathcal{O}_{\bar{v}}\not\subseteq Kw^{(p)}\\
&  \overset{\ref{Fact}}{\Longleftrightarrow} \forall x \in \mathfrak{m}_{v}\setminus\{0\}: \; 
1 + x^{-1}(1-\zeta_p)^p\mathcal{O}_{v} = 
1 + x^{-1}(\zeta_p-1)^p\mathcal{O}_{v}\not\subseteq (K^\times)^{p}
\end{align*}
holds and thus $v$ is indeed the canonical $p$-henselian valuation.
\end{proof}

Using the lemmas above, we can now prove the theorem:
\begin{proof}[Proof of the theorem]
We show first that $v_K^p$ is uniformly $\emptyset$-definable as such 
in case $p\neq 2$ or $Kv_K^p$ is not Euclidean.

As before, it suffices to characterize $v = v_K^p$
by a parameter-free first-order sentence in the language 
$\mathcal{L}_{val}:=\mathcal{L}_\textrm{ring} 
\cup \{\mathcal{O}_v\}$. Then by 
Beth's Definability Theorem 
(Theorem 5.5.4 in \cite{Hod97}), there is a
parameter-free
$\mathcal{L}_\textrm{ring}$-sentence $\phi_p(x)$ which defines $v_K^p$ as such.
Note that $Kv$ as an $\mathcal{L}_{ring}$-structure (respectively $vK$
as an $\mathcal{L}_{oag}$-structure with $\mathcal{L}_{oag}=\{+,<;0\}$)
is $\emptyset$-interpretable in the $\mathcal{L}_{val}$-structure $(K,v)$.
Thus, we may use $\mathcal{L}_{ring}$-elementary properties
of $Kv$ and $\mathcal{L}_{oag}$-properties of $vK$ in our
description of $v_K^p$.
The characterization is done by the following sentence $\psi_p$:
\begin{enumerate}[(i)] 
\item If $K=K(p)$, then $\mathcal{O}_v = K$
\item If $K \neq K(p)$, then
\begin{enumerate}[(1)]
\item $\mathcal{O}_v$ is a valuation ring of $K$
\item $v$ is $p$-henselian (see Theorem 1.5 in \cite{Koe95}): 
\begin{itemize}
\item if $\mathrm{char}(K)=p$, $$\mathfrak{m}_v \subseteq K^{(p)}:= 
\Set{x^p-x | x \in K}$$
\item if $\mathrm{char}(K)\neq p$, $$1 + p^2 \mathfrak{m}_v 
\subseteq (K^\times)^p \textrm{ and } \mathfrak{m}_v 
\subseteq \{x^p - x \mid x \in K\} + p\mathfrak{m}_v $$
\end{itemize}
\item if $Kv \neq Kv(p)$, then $Kv$ is not $p$-henselian (this is an 
$\mathcal{L}_{ring}$-elementary
property of $Kv$ by Corollary 2.2 in \cite{Koe95}) and, in case $p=2$,
not Euclidean 
\item if $Kv=Kv(p)$, then
\begin{itemize}
\item either $vK$ has no non-trivial $p$-divisible convex subgroup, this is
axiomatized by the $\mathcal{L}_{oag}$-sentence
$$ \forall \alpha \,( \alpha > 0 \rightarrow \exists \gamma\forall \delta \, 
(0 < \gamma \leq \alpha\, \wedge \, \gamma \neq \underbrace
{\delta+\dots +\delta}_{p\textrm{-times}}))$$
\item or it does and
\begin{itemize}
\item $\mathrm{char}(K)=p$ and $$\forall x \in \mathfrak{m}_v:\;
x^{-1}\mathcal{O}_v \not\subseteq K^{(p)}$$
\item or $(\mathrm{char}(K),\mathrm{char}(Kv))=(0,p)$ and $Kv$ is not perfect
\item or $(\mathrm{char}(K),\mathrm{char}(Kv))=(0,p)$ and $Kv$ is perfect
and  $$\forall x \in \mathfrak{m}_v:\;
1 + x^{-1}(\zeta_p-1)^p\mathcal{O}_v \not\subseteq (K^\times)^{p}.$$
\end{itemize}
\end{itemize}
\end{enumerate}
\end{enumerate}
It follows from Corollary \ref{1.1} and Lemmas \ref{1}, \ref{2} and \ref{3} 
that these conditions 
indeed hold for $v_K^p$ and that they furthermore guarantee $v = v_K^p$.

What is left to show is that if $K$ is a $p$-henselian field, 
with $p=2$ and
$Kv_2$ is Euclidean, then $v_K^p$ is not definable as such.

Consider
an $\omega$-saturated elementary extension 
$$(M,w) \succ (K,v_K^2)$$
in $\mathcal{L}_\textrm{ring} \cup 
\{\mathcal{O}_v\}$. Then the residue field $Mw$ is $\omega$-saturated and 
Euclidean, thus its unique ordering is non-archimedean. Hence
$Mw$ admits a non-trivial $2$-henselian valuation and so
$w \neq v_{M}^2$. In particular, $v_K^2$ is not definable as such.
\end{proof}

\section{Definable $t$-henselian valuations}
We now use our definitions of canonical $p$-henselian valuations to show that
in most cases a henselian valued field admits a definable
valuation which induces the (unique) henselian topology. As this topology
is $\emptyset$-definable in the language of rings, we will argue in the more
general context of $t$-henselian fields, namely fields which are elementarily
equivalent (in $\mathcal L_\textrm{ring}$) to some non-trivially 
henselian valued field. These were first introduced in \cite{PZ78}.

\begin{Def} 
Let $K$ be a field and $\tau$ a filter of neighbourhoods of $0$
on $K$. Then $(K, \tau)$ is called
\emph{$t$-henselian} if the following axioms hold, where $U$ and $V$ range
over elements of $\tau$ and $x,y$ range over elements of $K$:
\begin{enumerate}
\item $\forall U\;\{0\} \subsetneq U, \;\forall  x\neq 0 \,
\exists V\; x \notin V$
\item $\forall U\, \exists V \; V-V \subseteq U$
\item $\forall U\, \exists V \; V\cdot V \subseteq U$
\item $\forall U\,\forall x\, \exists V\; xV\subseteq U$
\item $\forall U\, \exists V\, \forall x,y\; (x \cdot y \in V 
\longrightarrow (x \in U \, \vee \,y \in U))$
\item (for every $n \in \mathbb N$) 
$\exists U\, \forall f \in X^{n+1} + X^n + U[X]^{n-1} \,\exists x \; f(x) = 0$
\end{enumerate}
Here, $U[X]^{m}$ denotes the set of polynomials with coefficients in $U$ and
degree at most $m$.
\end{Def}
 
Note that the first four axioms ensure that 
$\tau$ consists of a basis of neighbourhoods of $0$
of a non-discrete Hausdorff ring topology of $K$. The fifth axiom implies
that the topology is a $V$-topology and -- together with axioms (1)--(4) --
that it is in fact a field topology.
The final axiom scheme can be seen as a 
non-uniform version of henselianity.

Being $t$-henselian is an elementary property 
(in ${\mathcal L}_\textrm{ring}$): 
If $K$ is not separably closed, then
$K$ admits only one $t$-henselian topology and this topology is first-order
definable in the language of rings. Fix any irreducible, separable polynomial
$f \in K[X]$ with $\deg(f) >1$ and $a \in K$ satisfying $f'(a) \neq 0$. We
define
$$U_{f,a} \coloneqq \Set{f(x)^{-1} - f(a)^{-1} | x \in K }.$$
Then the sets $c \cdot U_{f,a}$ for $c \in K^\times$
form a basis of neighbourhoods of $0$ of the (unique)
$t$-henselian topology on $K$
(see \cite{Pr91}, p.\,203).
In particular, we get the following
\begin{Rem}[\cite{PZ78}, Remark 7.11] \label{topol}
If $K$ is not separably closed and admits a $t$-henselian topology, 
then
every field elementarily equivalent to $K$ carries a 
$t$-henselian topology.
\end{Rem}

Note that henselian fields are of course 
$t$-henselian with the topology being the valuation
topology induced by some (any) non-trivial henselian valuation. 
In the axiom scheme, we can choose $U$ as the maximal ideal of some
(any) non-trivial henselian valuation 
for any $n \in \mathbb N$.
If we take a $t$-henselian field, every sufficiently saturated elementarily
equivalent field will carry a henselian valuation:
\begin{Th}[\cite{PZ78}, Theorem 7.2]
 \label{sat}
Let $K$ be a non-separably closed field. Then $K$ is $t$-henselian iff 
$K$ is elementarily equivalent to some field admitting a non-trivial
henselian valuation.
\end{Th}

We now want to use the definability of $p$-henselian valuations to define
valuations on $t$-henselian fields. This improves the statement and the proof 
of Theorem 4.1 in
\cite{Koe94} in which a similar definition is found. 
\begin{Thm} \label{Thmt}
Let $K$ be a $t$-henselian field, neither separably closed nor real closed. 
Then $K$ admits a definable
valuation inducing the $t$-henselian topology.
If $K\neq K(p)$ holds for some prime $p$ or if 
$K$ has small absolute Galois group, i.e.~$K$ 
admits only finitely many Galois extensions of degree $n$ for each
natural number $n$,
then there is even a $\emptyset$-definable valuation inducing the 
$t$-henselian topology.
\end{Thm}
\begin{proof}
Note that without loss of generality, we may assume that $K$ is henselian.
By the previous remarks, 
any sufficiently saturated elementary extension $K'$ of $K$
is henselian. Furthermore, a base of the $t$-henselian topology can be defined 
using the same formulas on $K'$ as on $K$.
In our proof, 
we only need parameters to encode the minimal polynomial of
a specific Galois extension of $K'$.
As $K$ is relatively algebraically closed in $K'$, we can use parameters
from $K$. 
Thus, it suffices to give a 
definition of a valuation on $K'$ 
(using only parameters to encode some finite
Galois extension) inducing the
$t$-henselian topology. 
The same formula then also defines such a valuation on $K$.

First we assume that there is some $p$ with $K \neq K(p)$ (and $p\neq2$
if $K$ is Euclidean). Let $v$ be a non-trivial henselian valuation on $K$. 
In case
$\mathrm{char}(K) \neq p$ and 
$K$ does not contain a primitive $p$th root of unity,
consider $K(\zeta_p)$. Let $w$ be the
(by henselianity unique) extension of $v$ to $K(\zeta_p)$. 
Since $K(\zeta_p)$ is a finite Galois
extension of $K$ and the coefficients of the minimal polynomial
of this extension are all in $\mathrm{dcl}_K(\emptyset)$,
$K(\zeta_p)$ is interpretable without parameters in $K$.
Hence, it suffices to define a valuation on $K(\zeta_p)$ without parameters 
which induces
the same topology on $K(\zeta_p)$ as
 $w$. The
restriction of such a valuation to $K$ is then again $\emptyset$-definable and
induces the henselian topology on $K$. 
By Corollary \ref{Cmain}, some non-trivial coarsening of $v_{K(\zeta_p)}^p$
is $\emptyset$-definable on $K(\zeta_p)$.
As $v_{K(\zeta_p)}$ is in particular $p$-henselian, these valuations are
comparable and thus induce
the same topology. 

Otherwise, we have that $K = K(p)$ holds for all primes $p$ with $p \mid \#G_K$
(except possibly for $p=2$ if $K$ is Euclidean).
We may assume that $K$ is not Euclidean, since -- as above -- it suffices
to define a suitable valuation without parameters on $K(i)$.

Furthermore, there must be at least one prime 
$p$ with $p \mid \#G_K$,
else $K$ would be separably or real 
closed. Using parameters from $K$, we can now define some finite 
Galois extension 
$L$ of $K$ such that $L\neq L(p)$ holds. By the first part of the proof,
there is an $\emptyset$-definable valuation on $L$ inducing the
$t$-henselian toplogy, thus its restriction to $K$ is a definable valuation 
inducing the $t$-henselian topology on $K$. 

For the last part, assume that $G_K$ is small.
Let $n$ be an integer such that there exists a Galois extension $L$ of $K$,
$[L:K]=n$, with $L \neq L(p)$ and $\zeta_p \in L$ in case 
$\mathrm{char}(K)\neq p$.
Consider the valuation ring
\begin{align*} \mathcal{O}:= 
\prod \Big( \mathcal{O}_{v_L^{p}} \cap K \;\Big|\;
&K \subseteq L \textrm{ Galois}, \,[L:K]=n,\,L \neq L(p), \zeta_p \in L
\textrm{ in case }\mathrm{char}(K)\neq p \Big)
\end{align*}
on $K$. Note that since $G_K$ is small, this product is finite and
thus $\mathcal{O}$ is the valuation ring of 
the finest common coarsening of all the restrictions of $v_L^p$
to $K$. In particular, it induces the same topology as these restrictions, 
namely the $t$-henselian topology on $K$.
By Theorem \ref{main}, the ring $\mathcal{O}$ is $\emptyset$-definable. Thus,
this gives a non-trivial $\emptyset$-definable valuation on $K$ inducing the
$t$-henselian topology. 
\end{proof}

\section*{Acknowledgements} The authors would like to thank 
Katharina Dupont for finding a number of
mistakes in an earlier version.

\bibliographystyle{alpha}
\bibliography{franzi}
\end{document}